\documentclass[centertags,12pt]{amsart}
\usepackage{latexsym}
\usepackage{amsthm}
\usepackage{amssymb}
\usepackage{color}
\usepackage{xcolor}
\usepackage[utf8]{inputenc}
\usepackage{mathrsfs}
\usepackage[english]{babel}
\usepackage{setspace}
\usepackage{ textcomp }
\usepackage{graphicx}
\usepackage{cite}
\usepackage{cancel}
\usepackage{comment}
\graphicspath{ {images/} }

\numberwithin{equation}{section}
\makeatletter
\renewcommand{\subsection}{\@startsection
{subsection}{2}{0mm}{\baselineskip}{-0.25cm}
{\normalfont\normalsize\bf}}
\makeatother

\newtheorem{theorem}{Theorem}[section]
\newtheorem{proposition}[theorem]{Proposition}
\newtheorem{lemma}[theorem]{Lemma}
\newtheorem{corollary}[theorem]{Corollary}

   {\theoremstyle{definition}

}
   \theoremstyle{remark}
\newtheorem{remark}[theorem]{Remark}

\newtheorem{open}[theorem]{Open Problem}

\newcommand{\F}{{\mathbb F}}

\begin{document}

\author[D. Bartoli]{Daniele Bartoli}
\address{
Dipartimento di Matematica e Informatica, Universit\`a degli Studi di Perugia}
	\email{daniele.bartoli@unipg.it}
\author[M. Bonini]{Matteo Bonini}
\address{
Dipartimento di Matematica, Universit\`a di Trento}
	\email{matteo.bonini@unitn.it}
\author[M. Timpanella]{Marco Timpanella}
\address{Dipartimento di Matematica, Informatica ed Economia, Universit\`a degli Studi della Basilicata}
	\email{marco.timpanella@unibas.it}
\title{On the weight distribution of some minimal codes}
\begin{abstract}
Minimal codes are a class of linear codes which gained interest in the last years, thanks to their connections
to secret sharing schemes.
In this paper we provide the weight distribution and the parameters of families of minimal codes recently introduced by C. Tang, Y. Qiu, Q. Liao, Z. Zhou, answering some open questions.
\end{abstract}

\maketitle

{\bf Keyword}:  Linear code, minimal code, weight distribution.\\
{\bf 2010 MSC}: 94B05, 94A62
\section{Introduction}

A codeword $ c$ in a linear code $\mathcal{C}$ is called \emph{minimal} if its support (i.e., the set of nonzero coordinates of $ c $) does not contain the support of any other independent codeword. A minimal code is a linear code whose nonzero codewords are minimal. Minimal codewords and minimal codes in general have interesting connections to linear code-based secret sharing schemes (SSS); see \cite{Massey,Massey2}. 

A secret sharing scheme is a method to distribute shares of a secret to each of the participants $ \mathcal{P} $ in such a way that only the authorized subsets of $ \mathcal{P} $  could reconstruct the secret; see \cite{Shamir1979,Blakley1979}. 
%A set of participants $A$ is said to be a \emph{minimal authorized subset} if $A\in \Gamma$ and no proper subset of $A$ belongs to $\Gamma$. 

In  \cite{Massey,Massey2} Massey considered the use of linear codes for realizing a perfect (i.e. all authorized sets of participants can recover the secret while unauthorized sets of participants cannot determine any shares of the secret) and ideal (i.e. the shares of all participants are of the same size as that of the secret) SSS. It turns out that the access structure of the secret-sharing scheme corresponding to an $[n,k]_q$-code $\mathcal{C}$ is specified by the support of minimal codewords in the dual code $\mathcal{C}^{\bot}$ having $1$ as the first component.

In general, it is quite hard to find the whole set of  minimal codewords  of a given linear code; see \cite{BMeT1978,BN1990}. For this reason, minimal codes have been widely investigated in the last years; see for instance \cite{CCP2014,SL2012}.  Most of the known families of minimal codes are in characteristic two. 

A sufficient criterion for a linear code to be minimal is given by Ashikhmin and Barg in \cite{AB}.
\begin{lemma}\label{Th:Ashikhmin-Barg}
A linear code $\mathcal{C}$ over $\mathbb{F}_q$ is minimal if 
\begin{equation}\label{Eq:Ashikhmin-Barg}
\frac{w_{min}}{w_{max}} > \frac{q-1}{q},
\end{equation}
where $w_{min}$ and $w_{max}$ denote the minimum and maximum nonzero Hamming weights in $\mathcal{C}$,  respectively.
\end{lemma}

Families of minimal linear codes satisfying Condition \eqref{Eq:Ashikhmin-Barg} have been considered in several papers; e.g. see \cite{CDY2005,Ding2015,DLLZ2016,YD2006}. However, Condition~\eqref{Eq:Ashikhmin-Barg} is not necessary and examples of minimal codes not satisfying Condition~\eqref{Eq:Ashikhmin-Barg} have been constructed in i.e. \cite{CMP2013,CH2017,HDZ,DHZ,BarBon2019,BonBor2019,TQLZ,BBG,ABN19}.  

In this paper we provide the weight distribution and the parameters of families of minimal codes recently introduced in \cite{TQLZ}, answering to some  open questions.

The constructions of minimal codes presented in \cite{TQLZ} can be described in a geometrical way. Consider the affine space $\textrm{AG}(k,q)\simeq \mathbb{F}_q^k$ of dimension $k$ over the finite field $\mathbb{F}_q$, $q$ a prime power. 

Let $D=\{P_1,\ldots,P_n\}$ be a multiset of points in $\textrm{AG}(k,q)$ corresponding to the columns of a generator matrix of an $[n,k]_q$ linear code $\mathrm{C}_D$. For a hyperplane $H: \alpha_1 x_1+\cdots+ \alpha_k x_k=0$ through the origin  of $\textrm{AG}(k,q)$ and a point $P=(\overline{x}_1,\ldots,\overline{x}_k)\in \textrm{AG}(k,q)$, $H(P)$ denotes $\alpha_1 \overline{x}_1+\cdots+ \alpha_k \overline{x}_k$. 

With this notation, 
$$\mathrm{C}_D := \{(H(P_1),\ldots,H(P_n)) : H \text{ is an hyperplane of } AG(k,q) \text{ through the origin}\}.$$

The authors of  \cite{TQLZ}, following \cite{DLN,DN}, call $D$ the defining set of $\mathrm{C}_D$. They also present an interesting machinery which provides new minimal codes from old ones; see \cite[Theorem 43]{TQLZ}, where they make use of the concept of vectorial  cutting blocking set \cite{BonBor2019}.

\begin{theorem}
\label{codicitilde} Let $k \ge 2$. Let $D_1$ and $D_2$ be two vectorial cutting blocking sets in $AG(k,q)$ such that $D_1=a \cdot  D_1$ for any $a\in \F_q^*$.
Consider the following subset of $AG(k+1,q)$ 
\[
\widetilde{[D_1, D_2]}:=\left \{ (\mathbf x,1)\in AG(k+1,q): \mathbf x \in D_1 \right \} \bigcup \left \{ ( \mathbf x,0)\in AG(k+1,q): \mathbf x \in D_2 \right \}.
\]
Then, $\widetilde{[D_1, D_2]}$ is a vectorial cutting blocking set in $AG(k+1,q)$.
In particular, $\mathrm{C}_{\widetilde{[D_1, D_2]}}$ is a minimal code of length $(\# D_1+ \#D_2)$ and dimension $(k+1)$.
\end{theorem}

In \cite{TQLZ} the authors construct several families of minimal codes not satisfying Condition~\eqref{Eq:Ashikhmin-Barg}. They leave the determination of the weight distribution of some of them as open problems. In general, the computation of the weight distribution or of the weight spectrum (i.e. the set of its nonzero weights) of codes could be a challenging task. On the other hand, this computation provides important information, since for instance the weight distribution of a code allows the computation of the probability of error detection and correction with respect to some error detection and error correction algorithms; see \cite{Klove} for more details. 

Therefore our aim is to provide the weight spectrum or the weight distribution of specific minimal codes constructed in \cite{TQLZ}. In particular, we consider the following families of defining sets.
\begin{enumerate}
\item \label{Family:1} {\bf Family 1.} 
$$D_1 = \left\{(x_1, \ldots , x_k) \in AG(k,q)\setminus \{\overline{0}\} : \left(\sum_{i=1}^h x_i\right) \prod_{i=1}^h x_i=0 \right\},$$
where $4 \leq h \leq k$; see \cite[Open Problem 37]{TQLZ}.
\item \label{Family:2} {\bf Family 2.} 
$$D_2 = \left\{(x_1, \ldots , x_k) \in AG(k,q)\setminus \{\overline{0}\} :  \prod_{1\leq i<j\leq h}^h (x_i+x_j)=0 \right\},$$
where $3 \leq h \leq k$; see \cite[Open Problem 38]{TQLZ}.
\item \label{Family:3} {\bf Family 3.} 
$$D_3 = \left\{(x_1, \ldots , x_k) \in AG(k,q)\setminus \{\overline{0}\} : \prod_{i=1}^h x_i \prod_{1\leq i<j\leq h}^h (x_i+x_j)=0 \right\},$$
where $3 \leq h \leq k$; see \cite[Open Problem 39]{TQLZ}.
%\item \label{Family:4} {\bf Family 4.} 
%$$D_4 = \left\{(x_1, \ldots , x_k) \in \mathbb{F}_q^k\setminus \{0\} : \left(\sum _{i=1}^h x_i\right)\prod_{i=1}^h x_i \prod_{1\leq i<j\leq h}^h (x_i+x_j)=0 \right\},$$
%where $3 \leq h \leq k$; see \cite[Open Problem 40]{TQLZ}.
\item \label{Family:4} {\bf Family 4.} 
$$D_4 = \left\{(x_1, \ldots , x_k) \in AG(k,q)\setminus \{\overline{0}\} : \prod_{i=1}^h x_i =0 \right\},$$
where $3 \leq h \leq k$; see \cite[Open Problem 48]{TQLZ}.
\end{enumerate}

We determine the weight distribution of $\mathrm{C}_{D_1}$, $\mathrm{C}_{\widetilde{[D_1,D_1]}}$, and $\mathrm{C}_{\widetilde{[D_4,D_4]}}$, and the parameters of the codes  $\mathrm{C}_{D_2}$ and $\mathrm{C}_{D_3}$.

\section{Family 1}

By \cite[Theorem 23]{TQLZ} it is readily seen that the dimension of $\mathrm{C}_{D_1}$ is $k$. By \cite[Lemma 32]{TQLZ} and \cite[Theorem 33]{TQLZ}, $\mathrm{C}_D$ is a minimal code of length
\begin{equation*}
    n=q^{k-h-1}(q^{h+1}-(q-1)^{h+1}+(-1)^h(q-1))-1.
\end{equation*}

In order to compute the weight distribution of $\mathrm{C}_D$ it is useful to consider the following integers
\begin{eqnarray*}
\psi_{s}&:=&\#\left\{(x_1, \ldots , x_s) \in AG(s,q): \sum_{i=1}^s x_i =0 \text{ and } x_i \ne 0 \text{ for any } i=1,\ldots ,s\right\},\\
\varphi_{s}&:=&\#\left\{(x_1, \ldots , x_s) \in AG(s,q) : \sum_{i=1}^s x_i =1 \text{ and } x_i \ne 0 \text{ for any } i=1,\ldots ,s\right\}.
\end{eqnarray*}

As generalization of \cite[Lemma 31]{TQLZ}, we have 
\begin{equation*}
    \psi_{s}=\frac{(q-1)^s+(-1)^s(q-1)}{q}, \qquad     \varphi_{s}=\frac{(q-1)^s-\psi_{s}}{q-1}.
\end{equation*}

In particular note that $\psi_0=1$ and $\varphi_0=0$.

Consider now $a_1,\ldots,a_s\in \mathbb{F}_{q}^*$. It is readily seen that 
\begin{eqnarray*}
\psi_{s}&=&\#\left\{(x_1, \ldots , x_s) \in AG(s,q): \sum_{i=1}^s a_i x_i =0 \text{ and } x_i \ne 0 \text{ for any } i=1,\ldots ,s\right\},\\
\varphi_{s}&=&\#\left\{(x_1, \ldots , x_s) \in AG(s,q) : \sum_{i=1}^s a_i x_i =1 \text{ and } x_i \ne 0 \text{ for any } i=1,\ldots ,s\right\}.
\end{eqnarray*}

Let $\pi$ be the hyperplane of $AG(k,q)$ through the origin with affine equation 
\begin{equation}\label{iperpiano}
     a_{i_1}x_{i_1}+\ldots+ a_{i_s}x_{i_s}+b_{j_1}x_{j_1}+\ldots +b_{j_r}x_{j_r}=0,
\end{equation}
where $s\geq 0$, $r\geq 0$, $a_{i_1},\ldots,a_{i_s},b_{j_1},\ldots,b_{j_r}\in \mathbb{F}_q^*$, $i_1,\ldots,i_s \in \{1,\ldots,h\}$ and $j_1,\ldots,j_r \in \{h+1,\ldots,k\}$.

For the weight distribution of $\mathrm{C}_{D_1}$, we need to investigate the number of solutions $\Lambda$ of the system 
\begin{equation}\label{u'sistemone}
    \begin{cases} 
    
     a_{i_1}x_{i_1}+\ldots+ a_{i_s}x_{i_s}+b_{j_1}x_{j_1}+\ldots +b_{j_r}x_{j_r}=0 \\
     (x_1+\ldots+x_h) x_1\cdot \ldots \cdot x_h=0
    
    \end{cases}.
\end{equation}
Indeed, the weight of the codeword induced by $\pi$ is $n-\Lambda+1$.

\begin{proposition}\label{Prop:r>=1}
Let $r\geq 1$. Then 
$$\Lambda=q^{k-1}-q^{k-h-1}(q-1)^h+\psi_{h}q^{k-h-1}=q^{k-h-1}(q^h+\psi_{h}-(q-1)^h).$$
\end{proposition}
\begin{proof}
An easy computation shows that the number of solutions of the system 
\begin{equation*}
    \begin{cases} 
    
     a_{i_1}x_{i_1}+\ldots+ a_{i_s}x_{i_s}+b_{j_1}x_{j_1}+\ldots +b_{j_r}x_{j_r}=0 \\
     x_1\cdot \ldots \cdot x_h=0
    
\end{cases}
\end{equation*}
is $q^{k-1}-q^{k-h-1}(q-1)^h$.
Therefore it remains to compute the number of solutions of 
\begin{equation}\label{u'sistemone2}
\begin{cases} 
    
     a_{i_1}x_{i_1}+\ldots+ a_{i_s}x_{i_s}+b_{j_1}x_{j_1}+\ldots +b_{j_r}x_{j_r}=0 \\
     x_1+ \ldots + x_h=0\\
     x_1\cdot \ldots \cdot x_h\ne 0.
    
\end{cases}
\end{equation}
    Rearranging, the above system \eqref{u'sistemone2} is equivalent to 
    \begin{equation}\label{u'sistemone3}
\begin{cases} 
    
     x_{j_1}=-\alpha_{i_1}x_{i_1}-\ldots-\alpha_{i_s}x_{i_s}-\beta_{j_2}x_{j_2}-\ldots-\beta_{j_r}x_{j_r} \\
     x_1+ \ldots + x_h=0\\
     x_1\cdot \ldots \cdot x_h\ne 0,
    
\end{cases}
\end{equation}
with $\alpha_{i_l}=a_{i_l}/b_{j_1}$ and $\beta_{j_l}=b_{j_l}/b_{j_1}$. Since the number of solutions of \eqref{u'sistemone3} is $\psi_{h}q^{k-h-1}$, we obtain $\Lambda=q^{k-1}-q^{k-h-1}(q-1)^h+\psi_{h}q^{k-h-1}=q^{k-h-1}(q^h+\psi_{h}-(q-1)^h)$.
\end{proof}

\begin{proposition}\label{Prop:As}
Let $l\geq 1$, $r_1,\ldots, r_l\geq 1$, and consider $l$ pairwise distinct nonzero elements $\alpha_1,\ldots,\alpha_l$ of $\mathbb{F}_q$. The  number  of solutions  of the system 
\begin{equation}\label{SistemA2}
S_{r_1,\ldots,r_l}(\gamma) : \begin{cases}
    x^{(1)}_{1}+ \ldots + x^{(1)}_{r_1}+x^{(2)}_{1}+ \ldots + x^{(2)}_{r_2}+\cdots+x^{(l)}_{1}+ \ldots + x^{(l)}_{r_l}=\gamma\\
     \alpha_1(x^{(1)}_{1}+ \ldots + x^{(1)}_{r_1})+\alpha_2(x^{(2)}_{1}+ \ldots + x^{(2)}_{r_2})+\cdots+\alpha_l(x^{(l)}_{1}+ \ldots + x^{(l)}_{r_l})=0 \\
    \prod_{i,j}x^{(i)}_{j}\ne 0,
\end{cases}
\end{equation}
is, for $l=1$, $A_{r_1}=\psi_{r_1}$ if $\gamma=0$ and $A_{r_1}=0$ otherwise, and for $l>1$
\begin{equation}\label{NumberOfSolutions}
\left\{\begin{array}{ll}
   A_{r_1,\ldots,r_l}=\psi_{r_1+\cdots +r_{l-1}}\varphi_{r_l}+(-1)^{r_l}A_{r_1,\ldots,r_{l-1}},& \textrm{ if } \gamma=0;\\
  (\psi_{r_1+\cdots+r_{l}}-A_{r_1,\ldots,r_l})/(q-1),& \textrm{ if } \gamma\neq 0.\\
\end{array}
\right.
\end{equation}
\end{proposition}

\proof
We proceed by induction on $l$ and we also show that the number of solutions does not depend on the values $\alpha_1,\ldots,\alpha_l$.  If $\ell=1$, it is clear that if $\gamma\neq0$ then the number of solutions is $0$. Also, if $\gamma=0$, this number is precisely $\psi_{r_1}$. Clearly, this does not depend on the value $\alpha_1$.

Suppose that Formula \eqref{NumberOfSolutions} holds for $\ell\geq 1$ and  that the number of solutions does not depend on the values $\alpha_1,\ldots,\alpha_l$. Consider $\ell+1$. We first deal with $\gamma=0$. The system $S_{r_1,\ldots,r_l,r_{l+1}}(0)$ can be written as 
\begin{equation*}
S^{\prime}_{r_1,\ldots,r_l,r_{l+1}}(0) : 
\begin{cases}
    x^{(1)}_{1}+ \ldots + x^{(1)}_{r_1}+x^{(2)}_{1}+ \ldots + x^{(2)}_{r_2}+\cdots+x^{(l+1)}_{1}+ \ldots + x^{(l+1)}_{r_{l+1}}=0\\
     (\alpha_1-\alpha_{l+1})(x^{(1)}_{1}+ \ldots + x^{(1)}_{r_1})+(\alpha_2-\alpha_{l+1})(x^{(2)}_{1}+ \ldots + x^{(2)}_{r_2})\\
     \hspace*{5 cm}+\cdots+(\alpha_l-\alpha_{l+1})(x^{(l)}_{1}+ \ldots + x^{(l)}_{r_l})=0 \\
    \prod_{i,j}x^{(i)}_{j}\ne 0.
\end{cases}
\end{equation*}
Each solution of $S^{\prime}_{r_1,\ldots,r_l,r_{l+1}}(0)$ is a solution of precisely one of the following systems
\begin{equation*}
S^{\prime\prime}_{r_1,\ldots,r_l,r_{l+1}}(\gamma) : 
\begin{cases}
    x^{(1)}_{1}+ \ldots + x^{(1)}_{r_1}+x^{(2)}_{1}+ \ldots + x^{(2)}_{r_2}+\cdots+x^{(l)}_{1}+ \ldots + x^{(l)}_{r_l}=\gamma\\
    x^{(l+1)}_{1}+ \ldots + x^{(l+1)}_{r_{l+1}}=-\gamma\\
     (\alpha_1-\alpha_{l+1})(x^{(1)}_{1}+ \ldots + x^{(1)}_{r_1})+(\alpha_2-\alpha_{l+1})(x^{(2)}_{1}+ \ldots + x^{(2)}_{r_2})\\
     \hspace*{5 cm}+\cdots+(\alpha_l-\alpha_{l+1})(x^{(l)}_{1}+ \ldots + x^{(l)}_{r_l})=0 \\
    \prod_{i,j}x^{(i)}_{j}\ne 0.
\end{cases}
\end{equation*}
Viceversa, each solution of a particular $S^{\prime\prime}_{r_1,\ldots,r_l,r_{l+1}}(\gamma) $ is a solution of $S^{\prime}_{r_1,\ldots,r_l,r_{l+1}}(0)$. 

The number of solutions of $S^{\prime\prime}_{r_1,\ldots,r_l,r_{l+1}}(\gamma) $ is 
$A_{r_1,\ldots,r_l}\psi_{r_{l+1}}$ if $\gamma=0$, and 
$$ \frac{\psi_{r_1+\cdots+r_{l}}-A_{r_1,\ldots,r_l}}{q-1}\varphi_{r_{l+1}}$$
otherwise. By hypothesis these numbers do not depend on the choice of $\alpha_1-\alpha_{l+1},\ldots,\alpha_{l}-\alpha_{l+1}$. Summing up, the number of solutions of $S_{r_1,\ldots,r_l,r_{l+1}}(0)$ is 
\begin{eqnarray*}
A_{r_1,\ldots,r_l}\psi_{r_{l+1}}+(q-1)\frac{(\psi_{r_1+\cdots+r_{l}}-A_{r_1,\ldots,r_l})}{q-1}\varphi_{r_{l+1}}&=&\psi_{r_1+\cdots+r_{l}}\varphi_{r_{l+1}}+A_{r_1,\ldots,r_l}(\psi_{r_{l+1}}-\varphi_{r_{l+1}})\\
&=&\psi_{r_1+\cdots+r_{l}}\varphi_{r_{l+1}}+(-1)^{r_{l+1}}A_{r_1,\ldots,r_l}.\\
\end{eqnarray*}
It is readily seen that the number of solutions of $S_{r_1,\ldots,r_l,r_{l+1}}(\gamma)=S_{r_1,\ldots,r_l,r_{l+1}}(\delta)$ for any non-zero $\gamma,\delta\in \mathbb{F}_q$. The claim follows.
\endproof

\begin{remark}
From Proposition \eqref{Prop:As} it follows that 
$$A_{r_1,\ldots,r_l}=\psi_{r_1+\cdots +r_{l-1}}\varphi_{r_{l}}+(-1)^{r_{2}+\cdots +r_l}\psi_{r_1}+\sum_{i=1}^{l-2}(-1)^{r_{l-i+1}+\cdots +r_l}\psi_{r_1+\cdots +r_{l-i-1}}\varphi_{r_{l-i}}.$$
\end{remark}

As a notation, for $r_1,\ldots,r_l$ all distinct from $0$, we denote by $A_{r_1,\ldots,r_l,0}$ the integer $A_{r_1,\ldots,r_l}$.

\begin{proposition}\label{Prop:r=0}
Let $r=0$. Then the number of solutions of \eqref{u'sistemone} is 
$$\Lambda=q^{k-1}-(q-1)^{h-s}q^{k-h}\psi_{s} + q^{k-h}A_{r_1,\ldots,r_l,h-s}.$$
\end{proposition}
\proof

Without loss of generality we can assume $(i_1,\ldots,i_s)=(1,\ldots,s)$. 
As in Proposition \ref{Prop:r>=1} we count the number of solutions of two different systems, namely 

 \begin{equation}\label{System:1}
    \begin{cases} 
    
     a_{1}x_{1}+\ldots+ a_{s}x_{s}=0 \\
     x_1\cdot \ldots \cdot x_h=0
    
\end{cases}
\end{equation}
and 
 \begin{equation}\label{System:2}
\begin{cases} 
     x_{1}+ \ldots + x_{h}=0\\
     a_1x_{1}+\ldots+ a_{s}x_{s}=0 \\
    x_{1}\cdot \ldots \cdot x_{h}\ne 0.
\end{cases}
\end{equation}

In order to count the number of solutions of \eqref{System:1}, we consider 
 \begin{equation*}
    \begin{cases} 
    
     a_{1}x_{1}+\ldots+ a_{s}x_{s}=0 \\
     x_1\cdot \ldots \cdot x_h\ne 0.
 
\end{cases}
\end{equation*}
Here, we have $(q-1)^{h-s}q^{k-h}$ choices for $x_{s+1},\ldots, x_k$, while for the remaining coordinates we have $\psi_{s}$ possibilities: in total $(q-1)^{h-s}q^{k-h}\psi_{s}$ solutions.

This shows that System \eqref{System:1}
has $q^{k-1}-(q-1)^{h-s}q^{k-h}\psi_{s}$ solutions.

We now deal with System \eqref{System:2}.

We write \eqref{System:2} (up to a permutation of $(1,\ldots,s)$) in blocks of proportionality as 
\begin{equation}\label{System:2N}
\begin{cases} 
     x_{1}+ \ldots + x_{s}+x_{s+1}+\ldots+x_h=0\\
     \alpha_1(x_{1}+\ldots+x_{r_1})+\ldots + \alpha_l(x_{s-r_l+1}+\ldots+x_{s})=0 \\
    x_{1}\cdot \ldots \cdot x_{s}\cdot x_{s+1}\ldots \cdot x_{h}\ne 0
    \end{cases},
\end{equation}
for some $l\geq 1$, $r_1,\ldots, r_l\geq 1$ such that $r_1+\ldots+ r_l= s$, $\alpha_i$ pairwise distinct and nonzero. 

Note that if $s=h$ then the number of solutions of \eqref{System:2N} is  $q^{k-h}A_{r_1,\ldots,r_l}=q^{k-h}\psi_0 A_{r_1,\ldots,r_l}=q^{k-h}\psi_0 A_{r_1,\ldots,r_l,0}$.

Suppose now $s<h$. 
Each solution of \eqref{System:2N} is a solution of a certain 
\begin{equation}\label{Sistem:A}
S_{\gamma} : \begin{cases}
    x_{1}+ \ldots + x_{s}=\gamma\\
    x_{s+1}+\ldots+x_h=-\gamma\\
     \alpha_1(x_{1}+\ldots+x_{r_1})+\ldots + \alpha_l(x_{s-r_l+1}+\ldots+x_{s})=0 \\
    x_{1}\cdot \ldots \cdot x_{s}\cdot x_{s+1}\ldots \cdot x_{h}\ne 0.
\end{cases}
\end{equation}

By Proposition \ref{Prop:As}, for $\gamma=0$ System \eqref{Sistem:A} has $q^{k-h}\psi_{h-s}A_{r_1,\ldots,r_l}$ solutions, whereas for $\gamma\neq 0$, the number of solutions is $q^{k-h}\varphi_{h-s}(\psi_{r_1+\cdots+r_{l}}-A_{r_1,\ldots,r_l})/(q-1)$.
Summing up, the number of solutions of \eqref{System:2} is 

\begin{eqnarray*}
q^{k-h}\psi_{h-s}A_{r_1,\ldots,r_l}+q^{k-h}\varphi_{h-s}(\psi_{r_1+\cdots+r_{l}}-A_{r_1,\ldots,r_l})&=&q^{k-h}(\psi_{r_1+\cdots+r_{l}}\varphi_{h-s}+(-1)^{h-s}A_{r_1,\ldots,r_l})\\
&=&q^{k-h}A_{r_1,\ldots,r_l,h-s}
\end{eqnarray*}

The claim follows.
\endproof

Finally, we provide the weight spectrum and the weight distribution of the code $\mathrm{C}_{D_1}$ answering to \cite[Open Problem 37]{TQLZ}. 

For an $l$-tuple  $r_1,\ldots,r_l$, we say that it is of type $(i_1,\ldots,i_j)$ if there are $j$ distinct values among  $r_1,\ldots,r_l$ and they are repeated $i_1,\ldots,i_j$ times.

\begin{theorem}\label{weightDistr1}
The weight spectrum of the minimal code $\mathrm{C}_{D_1}$ is 
\begin{equation*}
    \left\{n-q^{k-h-1}(q^h+\psi_{h}-(q-1)^h)+1, n- q^{k-1}-(q-1)^{h-s}q^{k-h}\psi_{s} + q^{k-h}A_{r_1,\ldots,r_l,h-s}+1\right\},
\end{equation*}
where $s$ ranges in $1, \ldots, h$ and $r_1+\cdots+r_l=s$. 
Moreover, the number $B_i$ of codewords of weight $i$ is
\begin{itemize}
    \item[(i)] $q^k-q^h$, if $i=n-q^{k-h-1}(q^h+\psi_{h}-(q-1)^h)+1$;
    \item[(ii)] $\binom{h}{s}\binom{s}{r_1;\ldots; r_l}\binom{l}{i_1;\ldots; i_j}\binom{q-1}{l}$, if 
    $i=n- q^{k-1}-(q-1)^{h-s}q^{k-h}\psi_{s} + q^{k-h}A_{r_1,\ldots,r_l,h-s}+1 $ and $r_1,\ldots,r_l$ is of type $(i_1,\ldots,i_j)$.
    
\end{itemize}

\end{theorem}
\proof
The claim on the weight spectrum follows from Propositions \ref{Prop:r>=1} and \ref{Prop:r=0}.

Let $\bar{i}=n-q^{k-h-1}(q^h+\psi_{h}-(q-1)^h)+1$. By Proposition \ref{Prop:r>=1}, every hyperplane $H: \alpha_1 x_1+\ldots + \alpha_k x_k=0$ with $(\alpha_{h+1},\ldots, \alpha_{k}) \neq (0,\ldots, 0)$ induces a codeword of weight $\bar{i}$, whence $B_{\bar{i}}=q^k-q^h$.

Assume now $\bar{i}=n- q^{k-1}-(q-1)^{h-s}q^{k-h}\psi_{s} + q^{k-h}A_{r_1,\ldots,r_l,h-s}+1 $ for a partition $(r_1,\ldots,r_l)$ of $s$, $s \in [1,\ldots,h]$, $l\geq 1$, of type $i_1,\ldots,i_j$. 
We count the number of $k$-tuples of $(\mathbb{F}_q)^k$ such that the last $k-h$ entries are zero and that admit, among the first $h$ entries, $l$ distinct nonzero values and $h-s$ zeros.

 The $h-s$ zero entries can be chosen in $\binom{h}{s}$ ways among the first $h$ entries. The possible $l$-tuples of nonzero elements of $\mathbb{F}_q$ are $\binom{q-1}{l}$. Finally for any chosen $l$-tuple $\alpha_1\,\ldots,\alpha_l$, $\binom{s}{r_1;\ldots; r_l}\binom{l}{i_1;\ldots; i_j}$
counts the number $s$-uples where $\alpha_1,\ldots,\alpha_l$ appear exactly $r_1,\ldots,r_l$ times. 

Each hyperplane $H$ corresponding to such a $k$-tuple 
 induces,  by Proposition \ref{Prop:r=0}, a codeword of weight $\bar{i}$. 

\endproof

\begin{remark}\label{Remark1}
The weights in Theorem \ref{weightDistr1} (ii) are not all distinct. For instance, let $h=4$ and $k\geq h$. Then the weights corresponding to the choices $s=4$, $r_1=3$, $ r_2=1$ and $s=2$, $r_1=1$, $r_2=1$ are equal. 
\end{remark}

\section{Family 2}
By \cite[Theorem 23]{TQLZ} it is readily seen that the dimension of $\mathrm{C}_{D_2}$ is $k$.

\begin{proposition}\label{length:Family2}
Let 
$$\Gamma(h,q) := \sum_{s =1}^{\min(h,(q-1)/2)} \frac{(q-1) (q-3) \cdots (q-2s+1)}{s!}\mathcal{S}(h,s),$$
where $\mathcal{S}(x,y)$ is the number of surjective functions from a set of size $x$ to a set of size $y\leq x$.

The code $\mathrm{C}_{D_2}$ has length

\begin{equation*}
\left\{\begin{array}{ll}
q^{k-h}\left( q^h- q (q-1)  \cdots  (q-h+1) \right)-1,& \textrm{ if } p=2 \textrm{ and } h\leq q;\\
q^k-1 ,& \textrm{ if } p=2 \textrm{ and } h>q;\\
q^{k-h}\left(q^h-\Gamma(h,q)-h\Gamma(h-1,q) \right)-1,& \textrm{ if } p>2.
\end{array}
\right.
\end{equation*}

\end{proposition}
\proof
First, we count the number of $h$-tuples for which 
\begin{equation}\label{Condition1}
\textrm{no pairs of entries } (x_i,x_j), 1\leq i< j\leq h, \textrm{ satisfy } x_i+x_j=0.
\end{equation}

Assume  $p=2$. In this case the number of $h$-tuples for which at least one pair of entries $(x_i,x_j)$, $1\leq i< j\leq h$, satisfies $x_i+x_j=0$ is $q^h- q (q-1)  \cdots  (q-h+1)$ (in particular it is $q^h$ if $h>q$). 

From now on, let us consider the case $p>2$. We distinguish two cases.

\begin{itemize}
   
    \item[(1)] All entries are nonzero. Suppose that  the $h$ entries assume exactly $s$ distinct values $\alpha_1,\ldots,\alpha_s$ of $\mathbb{F}_q^*$. Since $\alpha_i\neq \pm \alpha _j$ for any $i\neq j$, $s$ can be at most $(q-1)/2$. For a given chosen number $s\in \{1,\ldots,\min \{h,(q-1)/2\}\}$, there are $(q-1)(q-3)\cdots (q-2s+1)/s!$ possible choices for the set  $\{\alpha_1,\ldots,\alpha_s\}$. In fact $\alpha_1$ can be chosen in $q-1$ ways, $\alpha_2\neq \pm \alpha_1$, $\alpha_3\notin \{\pm \alpha_1,\pm \alpha_2\}$ and so on. Now, when the set $\{\alpha_1,\ldots,\alpha_s\}$ is fixed, the $h$ entries can assume only values $\{\alpha_1,\ldots,\alpha_s\}$. The number of possible $h$-tuples equals the number $\mathcal{S}(h,s)$ of surjective functions from $\{1,\ldots,h\}$ to  $\{\alpha_1,\ldots,\alpha_s\}$. The number of $h$-tuples satisfying \eqref{Condition1} is $\Gamma(h,q)$.
     \item[(2)] One entry is $0$. In this case, any other entry is nonzero. To the other $h-1$  entries we can apply the same argument as above. Since the unique $0$ entry can appear in $h$ different positions, in this case the number of $h$-tuples satisfying \eqref{Condition1} is $h\Gamma(h-1,q)$.
     
\end{itemize}
Summing up, there are in total $\Gamma(h,q)+h\Gamma(h-1,q)$ $h$-tuples satisfying \eqref{Condition1}: the number of $h$-tuples for which at least one pair of entries $(x_i,x_j)$, $1\leq i< j\leq h$, satisfies $x_i+x_j=0$ is $q^h-\Gamma(h,q)-h\Gamma(h-1,q)$. 
The length of the code $\mathrm{C}_{D_2}$ is given by the number of $k$-tuples in $\mathbb{F}_q$ for which the first $h$ entries can be chosen in $q^h-\Gamma(h,q)-h\Gamma(h-1,q)$ ways.

\endproof

\begin{proposition}\label{pesi:family2}
Let  $q>5$ and $p>2$. Then the minimum weight in $\mathrm{C}_{D_2}$ is realized by the hyperplanes $x_i+x_j=0$, $1\leq i<j\leq h$.
\end{proposition}
\proof
It is readily seen that all the hyperplanes $x_i+x_j=0$,  $1\leq i<j\leq h$, contain $q^{k-1}-1$ points of $\mathcal{D}_2$ and therefore they correspond to minimum weight  codewords. Let $H$ be an hyperplane different from $x_i+x_j=0$, $1\leq i<j\leq h$. 
\begin{itemize}
    \item If $H : x_i+\alpha x_j=0$ for some $i\neq j$ and $\alpha \neq 1$ then the point $$(1,\ldots,1,\underbrace{-\alpha}_i,1,\ldots,1)\in H\setminus    \left(\mathcal{D}_2 \cup \{0\}\right)$$ and therefore $w(c_H)>n-q^{k-1}+1$, where $n$ is the length of $\mathrm{C}_{D_2}$. 
    \item Suppose now that $H: x_i=\beta x_l+\sum_{j\in J}\alpha_j x_j$, with $\#J\geq 1$, $\beta\neq 0,-1$, $l\notin J$. Let $\lambda \in \mathbb{F}_q\setminus \{-1,-(\sum_{j\in J}\alpha_j)/(\beta+1),-(1+\sum_{j\in J}\alpha_j)/\beta\}$. Then the point $$(1,\ldots,1,\underbrace{\lambda}_l,1,\ldots,1,\underbrace{\lambda\beta+\sum_{j\in J}\alpha_j}_i,1,\ldots,1)\in H\setminus \left(\mathcal{D}_2 \cup \{0\}\right)$$
    therefore $w(c_H)>n-q^{k-1}+1$. 
    \item Consider now the case $H: x_i=-x_l-x_s-\sum_{j\in J}x_j$, with $\#J\geq 0$. Let $\lambda \in \mathbb{F}_q\setminus \{-1,-\#J\}$, $\mu \in \mathbb{F}_q \setminus \{-1,-\lambda,-\#J,1-\lambda-\#J\}$.  Then the point $$(1,\ldots,1,\underbrace{\lambda}_l,1,\ldots,1,\underbrace{\mu}_s, 1,\ldots,1,\underbrace{-\lambda -\mu -\#J}_i,1,\ldots,1)\in H\setminus \left(\mathcal{D}_2 \cup \{0\}\right)$$
    therefore $w(c_H)>n-q^{k-1}+1$.
\end{itemize}
\endproof

\section{Family 3}

By \cite[Theorem 23]{TQLZ} it is readily seen that the dimension of $\mathrm{C}_{D_3}$ is $k$.

\begin{proposition}
Let 
$\Gamma(h,q)$ be defined as in Proposition \ref{length:Family2}.
The code $\mathrm{C}_{D_3}$ has length

\begin{equation*}
\left\{\begin{array}{ll}
q^{k-h}\left(q^h-(q-1) \cdots (q-h)\right)-1, & \textrm{ if } p=2 \textrm{ and } h\leq q+1; \\
q^{k}-1, & \textrm{ if } p=2 \textrm{ and } h> q+1;\\
q^{k-h}\left(q^h-\Gamma(h,q)\right)-1,& \textrm{ if } p>2.
\end{array}
\right.
\end{equation*}

\end{proposition}
\proof
First, we investigate the number of $h$-tuples for which 
\begin{equation}\label{Condition3}
\textrm{no coordinates are zero and no pairs } (x_i,x_j), 1\leq i< j\leq h, \textrm{ satisfy } x_i+x_j=0.
\end{equation}
Assume first $p>2$.
By the proof of Proposition \ref{length:Family2} (case (1)) the number of $h$-tuples satisfying \eqref{Condition3} equals $\Gamma(h,q)$. Therefore the number of $h$-tuples for which one coordinate is zero or at least one pair of entries $(x_i,x_j)$, $1\leq i< j\leq h$, satisfies $x_i+x_j=0$ is $q^h-\Gamma(h,q)$. The length of the code $\mathrm{C}_{D_3}$ is given by the number of $k$-tuples in $\mathbb{F}_q$ for which the first $h$ entries are chosen in such $q^h-\Gamma(h,q)$ ways. 

Assume now $p=2$. In this case the number
of $h$-tuples for which one coordinate is zero or at least one pair of entries $(x_i,x_j)$, $1\leq i< j\leq h$, satisfies $x_i+x_j=0$ is $q^h-(q-1) \cdots  (q-h)$ (in particular it is $q^h$ if $h>q+1$). The claim follows.
\endproof

\begin{proposition}
Let  $q>5$ and $p>2$. Then the minimum weight in $\mathrm{C}_{D_3}$ is realized by the hyperplanes $x_i+x_j=0$ and $x_i=0$, $1\leq i<j\leq h$.
\end{proposition}
\proof
Clearly, for $1\leq i<j\leq h$, any hyperplane $x_i+x_j=0$ or $x_i=0$ contains $q^{k-1}-1$ points of $\mathcal{D}_{3}$ and hence such hyperplanes correspond to minimum weight codewords. We will show that if $H$ is an hyperplane (through the origin) different from $x_i+x_j=0$, $1\leq i<j\leq h$, and $x_i=0$, $1\leq i\leq h$, then there exists a point $P\in H\setminus \mathcal{D}_3$. We argue as in the proof of Proposition \ref{pesi:family2}, the only difference being that $P$ must not have zero coordinates.
\begin{itemize}
    \item If $H : x_i+\alpha x_j=0$ for some $i\neq j$ and $\alpha \neq 0,1$ then the point $$(1,\ldots,1,\underbrace{-\alpha}_i,1,\ldots,1)\in H\setminus \left(\mathcal{D}_3 \cup \{0\}\right)$$ and therefore $w(c_H)>n-q^{k-1}+1$, where $n$ is the length of $\mathrm{C}_{D3}$. 
    \item Suppose now that $H: x_i=\beta x_l+\sum_{j\in J}\alpha_j x_j$, with $\#J\geq 1$, $\beta\neq 0,-1$, $l\notin J$. Let $\lambda \in \mathbb{F}_q\setminus \{0,-1,-(\sum_{j\in J}\alpha_j)/(\beta+1),-(\sum_{j\in J}\alpha_j)/\beta,-(1+\sum_{j\in J}\alpha_j)/\beta\}$. Then the point $$(1,\ldots,1,\underbrace{\lambda}_l,1,\ldots,1,\underbrace{\lambda\beta+\sum_{j\in J}\alpha_j}_i,1,\ldots,1)\in H\setminus \left(\mathcal{D}_3 \cup \{0\}\right)$$
    therefore $w(c_H)>n-q^{k-1}+1$. 
    \item Consider now the case $H: x_i=-x_l-x_s-\sum_{j\in J}x_j$, with $\#J\geq 0$. Let $\lambda \in \mathbb{F}_q\setminus \{0,-1,-\#J\}$, $\mu \in \mathbb{F}_q \setminus \{0,-1,-\lambda,-\#J,-\lambda-\#J,1-\lambda-\#J\}$.  Then the point $$(1,\ldots,1,\underbrace{\lambda}_l,1,\ldots,1,\underbrace{\mu}_s, 1,\ldots,1,\underbrace{-\lambda -\mu -\#J}_i,1,\ldots,1)\in H\setminus \left(\mathcal{D}_3 \cup \{0\}\right)
    $$
    therefore $w(c_H)>n-q^{k-1}+1$.
\end{itemize}

\endproof

\section{Family 4}
In this Section we deal with codes $\mathrm{C}_{\widetilde{[D, D]}}$ as defined in Theorem \ref{codicitilde}. Note that if $D$ is a subset of $\textrm{AG}(k,q)$ such that $aD=D$ for every $a \in \mathbb{F}_q^*$ and $\# D=n$, then
$$D=\mathbb{F}_q^* P_1 \cup \mathbb{F}_q^* P_2\cup \cdots \cup \mathbb{F}_q^* P_{n/(q-1)},$$ for some $P_1,\ldots,P_{n/(q-1)}\in D$.
Also observe that the weight of any codeword of $\mathrm{C}_{\widetilde{[D, D]}}$ is divisible by $q-1$.

Since the defining set of $\mathrm{C}_{\widetilde{[D, D]}}$ is 
$$\{(x_1,\ldots,x_k,0) : (x_1,\ldots,x_k) \in D\}\cup \{(x_1,\ldots,x_k,1) : (x_1,\ldots,x_k) \in D\}\subset \textrm{AG}(k+1,q),$$ it follows that for any hyperplane $H\subset \textrm{AG}(k+1,q)$ through the origin, the corresponding codeword   $c_H\in \mathrm{C}_{\widetilde{[D, D]}}$ can be written as
$(c_{H,0},c_{H,1})$
where 
$$c_{H,0}=(H(x_1,\ldots,x_k,0))_{x\in (\F_{q}^k)^*} \qquad \textrm{ and } \qquad c_{H,1}=(H(x_1,\ldots,x_k,1))_{x\in (\F_{q}^k)^*}.$$
Clearly,  $\mathrm{w}(c_H)=\mathrm{w}(c_{H,0})+\mathrm{w}(c_{H,1})$.

\begin{proposition}
\label{prop:dist_codicitilde}
For $H: \alpha_1 x_1+\cdots +\alpha_k x_k+\alpha_{k+1}x_{k+1}=0$, let $ \widetilde{H}: \alpha_1 x_1+\cdots +\alpha_k x_k=0$. If $c_{H}\in \mathrm{C}_{\widetilde{[D, D]}}$ and  $c_{\widetilde{H}}\in \mathrm{C}_D$ are the codewords corresponding to $H$ and $\widetilde{H}$ respectively, then
\begin{itemize}
    \item[(i)] If $\alpha_{k+1}=0$ then $\mathrm{w}(c_{H,1})=\mathrm{w}(c_{H,0})$ and  $\mathrm{w}(c_{H})=2\mathrm{w}(c_{\widetilde{H}})$.
    \item[(ii)] If $\alpha_{k+1}\neq 0$ then    $\mathrm{w}(c_{H})=n+\frac{q-2}{q-1}\mathrm{w}(c_{ \widetilde{H}})$.
\end{itemize}
\end{proposition}
\proof
Point {(i)} is clear.

Suppose now that $\alpha_{k+1}\neq 0$.  From the assumptions on $D$, for each $P=(x_1,\ldots,x_k)\in D$ and $a\in\F_q^*$, the point $aP$ is in $D$.
We distinguish two cases:
\begin{enumerate}
    \item[(a)] If $\widetilde H(P)=0$ then $\widetilde{H}(Q)=0$ for any $Q \in \mathbb{F}_q^* P$. In this case the entries corresponding to $\mathbb{F}_q^* P$ in $c_{H,0}$ are $0$, whereas those in $c_{H,1}$ are nonzero.
    \item[(b)] If $\widetilde H(P)\neq 0$ then  there  exists a unique value  $a\in\mathbb{F}_q^*$ such that $H((ax_1,\ldots,ax_k,1))=\widetilde{H}(aP)+\alpha_{k+1}=0$. In this case no entry corresponding to $\mathbb{F}_q^* P$ in $c_{H,0}$ is $0$, whereas exactly one in $c_{H,1}$ vanishes.
\end{enumerate}
Therefore 
\[
\mathrm{w}(c_{H,1})=n-\frac{\mathrm{w}(c_{ \widetilde{H}})}{q-1}
\]
and 
$$\mathrm{w}(c_{H,1})+\mathrm{w}(c_{H,0})=n-\frac{\mathrm{w}(c_{ \widetilde{H}})}{q-1}+\mathrm{w}(c_{ \widetilde{H}})=n+\frac{(q-2)\mathrm{w}(c_{ \widetilde{H}})}{q-1}.$$
\endproof

Proposition \ref{prop:dist_codicitilde} shows that the weight distribution of  $\mathrm{C}_{\widetilde{[D, {D}]}}$ is uniquely determined by the weight distribution of $\mathrm{C}_D$.
As a corollary of Proposition \ref{prop:dist_codicitilde} the following holds.

\begin{corollary}
\label{Prop:dist_pesi}
Let $A_i$ be the number of codewords of weight $i$ in $\mathrm{C}_D$. Then the weight spectrum of $\mathrm{C}_{\widetilde{[D, {D}]}}$ is
\[
\bigcup_{i=1}^n\left\{2i,n+\frac{q-2}{q-1}i : \, A_i\ne0\right\}.
\]
Moreover, if $B_i$ denotes the number of codewords of weight $i$ in $\mathrm{C}_{\widetilde{[D, {D}]}}$, then
\begin{equation}\label{B_i}
B_{i}=\eta_{\frac{i}{2}}A_{\frac{i}{2}}+(q-1)\eta_{(i-n)\frac{q-1}{q-2}}A_{(i-n)\frac{q-1}{q-2}}
\end{equation}
where
\[
\eta_s=\begin{cases}
     1,\quad\text{if}\, s\in\mathbb{Z};\\
     0,\quad\text{otherwise.}
\end{cases}
\]
\end{corollary}

In what follows we will focus on the computation of the weight distribution of the code $\mathrm{C}_{\widetilde{[D_4, {D_4}]}}$, where $D_4 = \left\{(x_1, \ldots , x_k) \in AG(k,q)\setminus \{\overline{0}\} : \prod_{i=1}^h x_i =0 \right\}$.
First, we report some information on $\mathrm{C}_{D_4}$ proved in \cite[Theorem 23]{TQLZ}.

\begin{proposition}\label{cd4}
$\mathrm{C}_{D_4}$ is a $\left[n,\,k\, ,n-q^{k-1}+1\right]_q$-code, where $n=q^{k-h}(q^h-(q-1)^h)-1$. Moreover,
\begin{itemize}
\item $\mathrm{C}_{D_4}$ has weight spectrum 
 $$\left\{n-q^{k-1}+q^{k-h-1}(q-1)^h+1,n-q^{k-1}+q^{k-h}(q-1)^{h-s}\psi_{s}+1\right\},$$ where $s=1,\ldots,h$.
\item If $A_i$ denotes the number of codewords of $\mathrm{C}_{D_4}$ of weight $i$, then 
\begin{equation*}
A_i=\left\{\begin{array}{ll}
  q^k-q^h,& \textrm{ if } i=n-q^{k-1}+q^{k-h-1}(q-1)^h+1;\\
  \binom{h}{s}(q-1)^{s},& \textrm{ if } i=n-q^{k-1}+q^{k-h}(q-1)^{h-s}\psi_{s}+1.\\
\end{array}
\right.
\end{equation*}
\end{itemize}
\end{proposition}

As a notation, let
\begin{eqnarray*}
w_s&=&n-q^{k-1}+q^{k-h}(q-1)^{h-s}\psi_{s}+1,\\
w&=&n-q^{k-1}+q^{k-h-1}(q-1)^h+1.
\end{eqnarray*}

Note that if $h=k$, $A_w=0$.

We are now in position to address \cite[Open Problem 48]{TQLZ}, providing the parameters and the weight distribution of the code $\mathrm{C}_{\widetilde{[D_4, {D_4}]}}$.

\begin{proposition}\label{Prop:Family4}
$\mathrm{C}_{\widetilde{[D_4, {D_4}]}}$ is a $\left[2n,k+1,n\right]_q$-code, where $n$ is the lenght of $\mathrm{C}_{D_4}$. Moreover, the weight spectrum of $\mathrm{C}_{\widetilde{[D_4, {D_4}]}}$ is 
\begin{itemize}
\item for $k>h$ \[
    \left\{0,n,2w_s,n+w_s\frac{q-2}{q-1},2w,n+w\frac{q-2}{q-1}\right\}_{s=1,\dots,h};
    \]
    
\item for $k=h$ 
\[\left\{0,n,2w_s,n+w_s\frac{q-2}{q-1}\right\}_{s=1,\dots,h}.
    \]
\end{itemize}
\end{proposition}
\proof
The claim on the weight spectrum is a consequence of Propositions \ref{prop:dist_codicitilde} and \ref{cd4}.
We only need to prove that the minimum weight of $\mathrm{C}_{\widetilde{[D_4, {D_4}]}}$ equals $n$. 
By Proposition \ref{prop:dist_codicitilde}, the only candidates as minimum weights are those arising from the minimum weight codewords in $\mathrm{C}_{D_4}$ and from the null word of $\mathrm{C}_{D_4}$ . Therefore, the minimum distance of  $\mathrm{C}_{\widetilde{[D_4, {D_4}]}}$ is
\[
\min\bigg(n,2(n-q^{k-1}+1),(n-q^{k-1}+1)\frac{q-2}{q-1}\bigg)=n.
\]
\endproof

Clearly there may be collisions between two weights in the weight spectrum of Proposition \ref{Prop:Family4}. In the next proposition we provide a deeper analysis of the weight distribution of $\mathrm{C}_{\widetilde{[D_4, {D_4}]}}$.

\begin{proposition}
\label{Prop:dist_pesi_D4}
Let $B_i$ be the number of codewords of $\mathrm{C}_{\widetilde{[D_4, D_4]}}$ of weight $i$. If $q>3$ the weight distribution of $\mathrm{C}_{\widetilde{[D_4, D_4]}}$ is given in Table \ref{Table_q>3}.
\begin{table}[h]
\caption{Weight Distribution of $\mathrm{C}_{\widetilde{[D_4, D_4]}}$ for $q>3$}\label{Table_q>3}
\begin{center}
\begin{tabular}{|c|c|}
\hline
Weight $i$ & $B_i$\\ \hline 
$0$ & $1$ \\ \hline 
$n$ & $q-1$ \\ \hline
$2w_s$, for $s=1,\ldots,h$ & $\binom{h}{s}(q-1)^s$ \\ \hline
$n+w_s\frac{q-2}{q-1}$, for $s=1,\ldots,h$ & $\binom{h}{s}(q-1)^{s+1}$ \\ \hline
$2w$ & $q^k-q^h$ \\ \hline
$n+w\frac{q-2}{q-1}$ & $(q^k-q^h)(q-1)$ \\ \hline
\end{tabular}
\end{center}
\end{table}
\end{proposition}
\proof
 
The claim follows by Corollary \ref{Prop:dist_pesi} and Proposition \ref{Prop:Family4}, after proving that there are no collisions between two weights in the weight spectrum of $\mathrm{C}_{\widetilde{[D_4, D_4]}}$. First, observe that while $w$ and $w_s$ are divisible by $q-1$, they are not divisble by $(q-1)^2$ (possibly with the only  exception of $w_h$). Indeed,
    \begin{equation}\label{congruenza_w}
    \frac{w}{q-1}\equiv\frac{w_s}{q-1}\equiv\frac{q^k-1}{q-1}-\frac{q^{k-1}-1}{q-1}\equiv q^{k-1}\not\equiv0\mod q-1
    \end{equation}
    for $s=1,\dots,h-1$, while
    if $s=h$
    \begin{equation}\label{congruenza_w_1}
        \frac{w_h}{q-1}\equiv q^{k-h-1}(q^h+(-1)^h) \mod q-1.
    \end{equation}
    We now consider all the possibile cases of collision between two weights.
 \begin{itemize}
 \item If $k>h+1$ it is readily seen that $n\ne 2w_s$ and $n\ne 2w$, since $n\equiv -1 \mod q$ whereas $w,w_s \equiv 0 \mod q$.
 If $k=h+1$ or $k=h$, a direct computation shows that $n<\min\{2w,2w_s\}$. Indeed,
\begin{eqnarray*}
2w_s-n&=&n-2q^{k-1}+2q^{k-h}(q-1)^{h-s}\psi_s+2\\
 &>& n-2q^{k-1}+1=q^k-q^{k-h}(q-1)^{h}-2q^{k-1}>0,
\end{eqnarray*}
for $q>3$. The same argument yields $n<2w$.
 \item  $n=n+w\frac{q-2}{q-1}$ or $n=n+w_s\frac{q-2}{q-1}$ is impossible for $q>3$.
 \item $w_s=w$ (which yields $2w_s=2w$ and $n+w_s\frac{q-2}{q-1}=n+w\frac{q-2}{q-1}$) implies
    \[
    (q-1)^{h-s}q^{k-h}\psi_s=(q-1)^hq^{k-h-1}
    \]
    that is 
    \[
    q\psi_s=(q-1)^s
    \]
    a contradiction.
 
 \item As observed above, since $(q-1)$ divides $w$ and $w_s$ but $(q-1)^2$ does not (except possibly for $w_h$), we have 
 \[
    2w_s\ne n+\frac{q-2}{q-1}w
    \]
    for $s=1,\dots,h$, and
 \[
    2w\ne n+\frac{q-2}{q-1}w_{s}
    \]
    for $s=1,\dots,h-1$.
    
    It remains to check if it is possible that  $2w= n+\frac{q-2}{q-1}w_{h}$. Note first that if $h$ is odd then $\frac{w_h}{q-1} \not\equiv0\mod q-1$, whence the same argument as above applies and $2w\neq n+\frac{q-2}{q-1}w_{h}$. Assume now $h$ even. Then $2w= n+\frac{q-2}{q-1}w_{h}$ reads
    \[
    2(n-q^{k-1}+q^{k-h-1}(q-1)^h+1)=n+(q-2)(q^{k-h-1}(q^h-(q-1)^h)+q^{k-h-1}),
    \]
    that is
 \begin{equation}\label{differenza}
   n=q^{k}-q^{k-h}(q-1)^h+q^{k-h}-2q^{k-h-1}-2,
    \end{equation}
    a contradiction to $n\equiv -1 \mod q$.

\item If $w_s=w_{s^\prime}$ for some $s,s^{\prime}\in \{1,\ldots,h\}$ with  $s^\prime>s$, then 
\[ (-1)^{s^{\prime}}(q-1)^{h-s^{\prime}+1} = (-1)^{s}(q-1)^{h-s+1}
    \]
that is
\[
    (-1)^{s-s^\prime}(q-1)^{s^\prime-s}=1
    ,\]
a contradiction to $q>3$.
  
\item If $2w=n+w\frac{q-2}{q-1}$, then $(q-1)n=qw$; impossible since $n$ is not divisible by $q$. The same argument also shows that $2w_s\ne n+w_s\frac{q-2}{q-1}$.
\end{itemize}    
\endproof

\begin{remark}
If $q=3$, almost the same argument in Proposition \ref{Prop:dist_pesi_D4} applies: the only difference arises from Equation \eqref{differenza} when $k=h+1$. Indeed, in this case, $2w=n+\frac{q-2}{q-1}w_h$.  Table \ref{Table_q3} shows the weight distribution of $\mathrm{C}_{\widetilde{[D_4, D_4]}}$ for $q=3$ and $k=h+1$.
\end{remark}
\begin{center}

\begin{table}[h]
\caption{Weight Distribution of $\mathrm{C}_{\widetilde{[D_4, D_4]}}$ for $q=3$ and $k=h+1$}\label{Table_q3}
\begin{center}
\begin{tabular}{|c|c|}
\hline
Weight $i$ & $B_i$\\ \hline 
$0$ & $1$ \\ \hline 
$n$ & $2$ \\ \hline
$2w_s$, for $s=1,\ldots,k-1$ & $\binom{k-1}{s}2^s$ \\ \hline
$n+w_s/2$, for $s=1,\ldots,k-2$ & $\binom{k-1}{s}2^{s+1}$ \\ \hline
$2w$ & $3^k-3^{k-1}$ \\ \hline
$n+w/2$ & $2(3^k-3^{k-1})$ \\ \hline
\end{tabular}
\end{center}
\end{table}

\end{center}

As an application of Proposition \ref{prop:dist_codicitilde} we provide  the weight distribution of the code $\mathrm{C}_{|\widetilde{D_1,D_1}|}$. In this case we will not deal with possible collisions of two weights (since this problem is already hard to study for the weight distribution of $\mathrm{C}_{D_1}$).

\begin{proposition}
With the same notation as in Theorem \ref{weightDistr1}, the weight distribution of $\mathrm{C}_{\widetilde{[D_1, D_1]}}$ is given in Table \ref{Table_D1D1}.
\end{proposition}

\begin{table}[h]
\caption{Weight Distribution of $\mathrm{C}_{\widetilde{[D_1, D_1]}}$}\label{Table_D1D1}
\begin{center}
\begin{tabular}{|c|c|}
\hline
Weight $i$ & $B_i$\\ \hline 
$0$ & $1$ \\ \hline 
$n$ & $q-1$ \\ \hline
$2(n-q^{k-h-1}(q^h+\psi_{h}-(q-1)^h)+1)$ & $q^k-q^h$ \\ \hline
$n+(n-q^{k-h-1}(q^h+\psi_{h}-(q-1)^h)+1)\frac{q-2}{q-1}$ & $(q-1)(q^k-q^h)$ \\ \hline
$2(n- q^{k-1}-(q-1)^{h-s}q^{k-h}\psi_{s} + q^{k-h}A_{r_1,\ldots,r_l,h-s}+1)$ & $\binom{h}{s}\binom{s}{r_1;\ldots; r_l}\binom{l}{i_1;\ldots; i_j}\binom{q-1}{l}$ \\ \hline
$n+(n- q^{k-1}-(q-1)^{h-s}q^{k-h}\psi_{s} + q^{k-h}A_{r_1,\ldots,r_l,h-s}+1)\frac{q-2}{q-1}$ & $\left(q-1\right)\left(\binom{h}{s}\binom{s}{r_1;\ldots; r_l}\binom{l}{i_1;\ldots; i_j}\binom{q-1}{l}\right)$ \\ \hline
$n+(n- q^{k-1}-(q-1)^{h-s}q^{k-h}\psi_{s} + q^{k-h}A_{r_1,\ldots,r_l,h-s}+1)\frac{q-2}{q-1}$ & $\left(q-1\right)\left(\binom{h}{s}\binom{s}{r_1;\ldots; r_l}\binom{l}{i_1;\ldots; i_j}\binom{q-1}{l}\right)$ \\ \hline
\end{tabular}
\end{center}
\end{table}

Finally, we present the following open problems.
\begin{open}
Determine the weight distribution (without collisions) of $\mathrm{C}_{D_1}$ and $\mathrm{C}_{[\widetilde{D_1,D_1}]}$.
\end{open}
\begin{open}
Determine the weight distribution of $\mathrm{C}_{D_2}$ and $\mathrm{C}_{D_3}$.
\end{open}

\section{Acknowledgments*}
The research of D. Bartoli, M. Bonini, and M. Timpanella was partially supported  by the Italian National Group for Algebraic and Geometric Structures and their Applications (GNSAGA - INdAM).


\begin{thebibliography}{99}

\bibitem{ABN19} G. N. Alfarano, M. Borello, A. Neri.
\newblock A geometric characterization of minimal codes and their asymptotic performance.
\newblock arXiv:1911.11738 (2019).

\bibitem{AB} A. Ashikhmin, A. Barg.
\newblock Minimal vectors in linear codes.
\newblock \emph{IEEE Trans. Inf. Theory} {\bf 44}(5),  2010--2017 (1998).

\bibitem{BMeT1978} 
\newblock E. R Berlekamp, R. J. McEliece,  H. C. A. van Tilborg.
\newblock On the Inherent Intractability of Certain Coding Problems.
\newblock \emph{IEEE Trans. Inform. Theory} {\bf 24}(3), 384--386 (1978).

\bibitem{Blakley1979}  G. R. Blakley. 
\newblock Safeguarding cryptographic keys.
\newblock In: \emph{Proc. of AFIPS National Computer Conference}, New York, USA,  pp. 313--317 (1979).

\bibitem{BarBon2019} D. Bartoli, M. Bonini.
\newblock Minimal linear codes in odd characteristic.
\newblock \emph{IEEE Trans. Inf. Theory} {\bf 65}(7), 4152--4155 (2019).

\bibitem{BonBor2019} M. Bonini, M. Borello. 
\newblock Minimal linear codes arising from blocking sets. 
\newblock \emph{Journal of Algebraic Combinatorics}. https://link.springer.com/article/10.1007/s10801-019-00930-6.

\bibitem{BBG} D. Bartoli, M. Bonini, B. G\H{u}nes. 
\newblock An inductive construction of minimal codes.
\newblock arXiv:1911.09093 (2019).

\bibitem{BN1990} 
\newblock J. Bruck, M. Naor.
\newblock The Hardness of Decoding Linear Codes with Preprocessing.
\newblock \emph{IEEE Trans. Inform. Theory} {\bf 36}(2), 381--385  (1990).

\bibitem{CH2017}
\newblock S. Chang, J. Y. Hyun.
\newblock Linear codes from simplicial complexes.
\newblock \emph{Des. Codes Cryptogr.} {\bf 86}(10),  2167--2181 (2018).

\bibitem{CCP2014}
\newblock H. Chabanne, G. Cohen,  A. Patey.
\newblock Towards Secure Two-Party Computation from the Wire-Tap Channel. 
\newblock In: \emph{Information Security and Cryptology -- ICISC 2013}, Heidelberg, Germany,  2014, pp. 34--46.

\bibitem{CMP2013} 
\newblock G. D. Cohen, S. Mesnager,  A. Patey.
\newblock On minimal and quasi-minimal linear codes.
\newblock In: \emph{IMACC 2013}, Heidelberg, Germany,  2013, pp. 85--98.

\bibitem{CDY2005} 
\newblock C. Carlet, C. Ding,  J. Yuan.
\newblock Linear codes from highly nonlinear functions and their secret sharing schemes.
\newblock \emph{IEEE Trans. Inf. Theory} {\bf 51}(6),  2089--2102 (2005).

\bibitem{Ding2015} 
\newblock C. Ding.
\newblock Linear codes from some $2$-designs.
\newblock \emph{IEEE Trans. Inf. Theory} {\bf 60}(6), 3265--3275 (2015).

\bibitem{DHZ} C. Ding, Z. Heng,  Z. Zhou.
\newblock Minimal binary linear codes.
\newblock \emph{IEEE Trans. Inf. Theory} 64(10),  6536--6545 (2018).

\bibitem{DLLZ2016}
\newblock C. Ding, N. Li, C. Li,  Z. Zhou.
\newblock Three-weight cyclic codes and their weight distributions.
\newblock \emph{Discrete Math.} {\bf 339}(2), 415--427 (2016).

\bibitem{DLN} C. Ding, J. Luo, H. Niederreiter.
\newblock Two weight codes punctured from irreducible cyclic codes.
\newblock In: Li, Y., Ling, S., Niederreiter, H., Wang, H., Xing, C., Zhang, S. (Eds.) Proc. of the First InternationalWorkshop on Coding Theory and Cryptography, pp. 119--124. Singapore,World Scientific, (2008).

\bibitem{DN} C. Ding, H. Niederreiter.
\newblock Cyclotomic linear codes of order 3.
\newblock \emph{IEEE Trans. Inf. Theory} {\bf 53}, 2274--2277 (2007).

\bibitem{HDZ} Z. Heng, C. Ding,  Z. Zhou.
\newblock Minimal linear codes over finite fields.
\newblock \emph{Finite Fields Appl.} {\bf 54},  176--196 (2018).

\bibitem{Klove} T. Kl\o ve.
\newblock Codes for Error Detection.
\newblock Singapore: World Scientific, 2007.

\bibitem{Massey} J. L. Massey.
\newblock Minimal codewords and secret sharing. 
\newblock In: Proc. 6th Joint Swedish-Russian Int. Workshop on Info. Theory, Sweden, pp. 276--279 (1993).

\bibitem{Massey2} J. L. Massey.
\newblock Some applications of coding theory in cryptography. 
\newblock In: Codes and Cyphers: Cryptography and Coding IV, Esses, England, pp. 33--47 (1995).

\bibitem{Shamir1979} A. Shamir. 
\newblock How to share a secret.
\newblock \emph{Commun. ACM}  {\bf 22}(11), 612--613  (1979).

\bibitem{SL2012} 
\newblock Y. Song, Z. Li, Y. M. Li.  
\newblock Secret sharing with a class of minimal linear codes.
\newblock Acta Electronic Sinica {\bf 41}, 220--226 (2013).

\bibitem{TQLZ} C. Tang, Y. Qiu, Q. Liao, Z. Zhou.
\newblock Full Characterization of Minimal Linear Codes as Cutting Blocking Sets.
\newblock arXiv:1911.09867 (2019)

\bibitem{YD2006}
\newblock J. Yuan, C. Ding.
\newblock Secret sharing schemes from three classes of linear codes.
\newblock \emph{IEEE Trans. Inf. Theory} {\bf 52}(1), 206--212 (2006).
\end{thebibliography}
\end{document}